\documentclass[a4paper,11pt,reqno]{amsart}

\usepackage{amsmath,amssymb,amsthm}
\usepackage[colorlinks=true]{hyperref}
\usepackage{pictures}
\usepackage{graphicx}
\usepackage[margin=3.5cm]{geometry}

\newtheorem{theo}{Theorem}[section]
\newtheorem{lem}[theo]{Lemma}
\newtheorem{prop}[theo]{Proposition}

\theoremstyle{definition}
\newtheorem{defi}{Definition}[section]

\DeclareMathOperator{\TSPP}{TSPP}
\DeclareMathOperator{\sgn}{sgn}
\DeclareMathOperator{\diag}{diag}
\DeclareMathOperator{\inv}{inv}

\DeclareMathOperator{\asym}{\mathbf{ASym}}

\newcommand{\A}{\mathcal{A}}

\numberwithin{equation}{section}

\title[Alternating sign matrices and totally symmetric plane partitions]{Alternating sign matrices and totally symmetric plane partitions}

\author[F. Aigner]{Florian Aigner}
\address{Florian Aigner, LaCIM, Universit\'{e} du Qu\'{e}bec \`{a} Montr\'{e}al, Canada}
\email{florian.aigner@univie.ac.at}
\urladdr{https://homepage.univie.ac.at/florian.aigner/}

\author[I. Fischer]{Ilse Fischer}
\address{Fakult\"at f\"ur Mathematik, Universit\"at Wien, Austria}
\urladdr{https://www.mat.univie.ac.at/$\sim$ifischer/}

\author[M. Konvalinka]{Matja\v{z} Konvalinka}
\address{Fakulteta za matematiko in fiziko, Univerza v Ljubljana, \& In\v{s}titut za matematiko, fiziko in mehaniko, Ljubljana, Slovenija}
\urladdr{http://www.fmf.uni-lj.si/$\sim$konvalinka/}

\author[P. Nadeau]{Philippe Nadeau}
\address{Univ Lyon, CNRS, Universit\'e Claude Bernard Lyon 1, UMR 5208, Institut Camille Jordan, France}
\urladdr{http://math.univ-lyon1.fr/$\sim$nadeau/}

\author[V. Tewari]{Vasu Tewari}
\address{Department of Mathematics, University of Pennsylvania, USA}
\urladdr{https://www.math.upenn.edu/$\sim$vvtewari/}

\thanks{The first author acknowledges support from the Austrian Science Foundation FWF: J 4387 and SFB grant F50 and by the project ‘‘Austria/France Scientific \& Technological Cooperation’’ (BMWFW Project No. FR 10/2018 and PHC Amadeus 2018 Project No. 39444WJ), the second author acknowledges support from the Austrian Science Foundation FWF, SFB grant F50, and the third author acknowledges the financial support from the Slovenian Research Agency (research core funding No. P1-0294).}


\begin{document}

\begin{abstract}
We study the Schur polynomial expansion of a family of symmetric polynomials related to the refined enumeration of alternating sign matrices with respect to their inversion number, complementary inversion number and the position of the unique $1$ in the top row. 
We prove that the expansion can be expressed as a sum over totally symmetric plane partitions and we are also able to determine the coefficients.
This establishes a new connection between alternating sign matrices and a class of plane partitions, thereby complementing the fact that  
alternating sign matrices are equinumerous with totally symmetric self-complementary plane partitions as well as with descending plane partitions. As a by-product we obtain an interesting map from totally symmetric plane partitions to Dyck paths. The proof is based on a new, quite general antisymmetrizer-to-determinant formula.
\end{abstract}

\maketitle

\section{Introduction}

\emph{Plane partitions} were first studied by MacMahon \cite{MacMahon97} at the end of the 19th century, however found broader interest in the combinatorial community starting in the second half of the last century. \emph{Alternating sign matrices} (ASMs) on the other hand were introduced by Robbins and Rumsey \cite{RobbinsRumsey86} in the early 1980s. Together with Mills \cite{MillsRobbinsRumsey82}, they conjectured that the number of $n \times n$ ASMs is given by $\prod_{i=0}^{n-1} \frac{(3i+1)!}{(n+i)!}$. Stanley then pointed out to them that these numbers had appeared before in the work of Andrews \cite{Andrews79} as the enumeration formula for a certain class of plane partitions, called \emph{descending plane partitions}. Soon after that Mills, Robbins and Rumsey \cite{MillsRobbinsRumsey86} observed (conjecturally) that this formula also counts another class of plane partitions, namely \emph{totally symmetric self-complementary plane partitions}. Although these conjectures have all been proved since then, see among others \cite{Andrews94,Zeilberger96}, it is mostly agreed that there is no good combinatorial understanding of this relation between ASMs and certain classes of plane partitions since we lack combinatorial proofs of these results. The purpose of this paper is to relate ASMs to yet another class of plane partitions, namely \emph{totally symmetric plane partitions} (TSPPs), in a new way. This relation is via a certain Schur polynomial expansion. Other known relations between ASMs and TSPPs are via posets, see \cite[Section 8]{Striker11}, and the fact, that the number of symmetric plane partitions inside an $(n,n,n)$-box is the product of the number of TSPPs inside an $(n,n,n)$-box and the number of ASMs of size $n$, see \cite{Fischer05}, although it can be argued that the latter relation is in a sense more between TSPPs and TSSCPPs.

More concretely, the following symmetric functions are studied in this paper
\[
\A_n(u,v;\mathbf{x}):= \frac{\asym_{x_1,\ldots,x_n} 
\left[ \prod\limits_{i=1}^n x_i^{i-1} \prod\limits_{1 \le i < j \le n} (v+ (1-u-v) x_i + u x_i x_j) \right]}{\prod\limits_{1 \le i < j \le n} (x_j - x_i)}, 
\]
where $\asym$ denotes the antisymmetrizer, i.e.,
$
\asym_{x_1,\ldots,x_n}  f(\mathbf{x}) = \sum\limits_{\sigma \in {\mathcal S}_n} \sgn(\sigma) \allowbreak \cdot f(x_{\sigma(1)},\ldots,x_{\sigma(n)})
$
and $\mathbf{x}=(x_1,\ldots,x_n)$.
These symmetric functions have arisen in the special case $t^{\binom{n}{2}}\A_n\left(\frac{1}{t},\frac{1}{t};\mathbf{x}\right)$ in the work of Fischer and Riegler \cite[Corollary 10.2]{FischerRiegler15}, where the following connection to 
ASMs was proved.
\begin{theo} \label{thm:FiRie}
The number of $n \times n$ ASMs that have the unique $1$ in the top row in column $i$ and with $m$ occurrences of $-1$ is the 
coefficient of $z^{i-1} t^m$ in $t^{\binom{n}{2}}\A_n\left(\frac{1}{t},\frac{1}{t};z,1,\ldots,1\right)$.
\end{theo}

As remarked in \cite[Remark~2.1]{Fischer18}, the above result can be generalized as follows\footnote{It is a generalization since, for a given $n \times n$ ASM, the number of $-1$'s and the two inversion numbers sum to $\binom{n}{2}$.}.
\begin{theo} \label{thm: Fi18}
The number of $n \times n$ ASMs that have the unique $1$ in the top row in column $i$ and with inversion number $a$ and complementary inversion number $b$ is the coefficient of $u^a v^b z^{i-1}$ in $\A_n(u,v;z,1,\ldots,1)$.
\end{theo}

The main result of this paper is the following Schur polynomial expansion of these functions. It was conjectured independently by the first author together with Fran\c{c}ois Bergeron and the last three authors.  All notions are explained in Section~\ref{sect:notion}.

\begin{theo}
\label{thm: main result}
For all positive integers $n$, we have
\begin{equation}
\label{eq: A-poly as sum of TSPPs}
\A_n(u,v;\mathbf{x}) = \sum_{T \in \TSPP_{n-1}}\omega_{\pi(T)}(u,v)s_{\pi(T)}(\mathbf{x}),
\end{equation}
where $\pi(T)=(a_1,\ldots,a_l|b_1+1,\ldots,b_l+1)$ is the modified balanced partition associated to the TSPP $T$ in Frobenius notation and 
$\omega_{\pi(T)}(u,v)=u^{\sum_{i=1}^l (a_i+1)} (1-u-v)^{\sum_{i=1}^l(b_i-a_i)}v^{\binom{n}{2}-\sum_{i=1}^lb_i}$.
\end{theo}

For $n=3$ the right hand side of \eqref{eq: A-poly as sum of TSPPs} is a sum over all totally symmetric plane partitions inside a $(2,2,2)$-box, which are shown below. Theorem \ref{thm: main result} then states
\begin{multline*}
\A_n(u,v;x_1,x_2,x_3)= v^3 + u v^2 s_{1,1}(x_1,x_2,x_3)
\\+u(1-u-v)v s_{1,1,1}(x_1,x_2,x_3) + u^2 v s_{2,1,1}(x_1,x_2,x_3) + u^3 s_{2,2,2}(x_1,x_2,x_3).
\end{multline*}

\begin{center}
	\begin{tikzpicture}
	\begin{scope}[scale=0.42]
		\node at (-1.5,0) {$T$:};
		\node at (2,0) {$\emptyset$};
		\begin{scope}[xshift=6cm]
			\PlanePartition{{1,0},{0,0}}
		\end{scope}
		\begin{scope}[xshift=12cm]
			\PlanePartition{{2,1},{1,0}}
		\end{scope}
		\begin{scope}[xshift=18cm]
			\PlanePartition{{2,2},{2,1}}
		\end{scope}
		\begin{scope}[xshift=24cm]
			\PlanePartition{{2,2},{2,2}}
		\end{scope}
	\end{scope}
	\begin{scope}[scale=0.42, yshift=-7.5cm,]
		\node at (-1.5,2) {$\pi(T)$:};
		\node at (2,2) {$\emptyset$};
		\begin{scope}[xshift=4.5cm]
			\Tableaux{{1,1}}
		\end{scope}
		\begin{scope}[xshift=10.5cm]
			\Tableaux{{1,1,1}}
		\end{scope}
		\begin{scope}[xshift=16cm]
			\Tableaux{{2,1,1}}
		\end{scope}
		\begin{scope}[xshift=22cm]
			\Tableaux{{2,2,2}}
		\end{scope}
	\end{scope}
	\begin{scope}[scale=0.42, yshift=-9cm]
		\node at (-1.5,0) {$\omega_{\pi(T)}(u,v))$:};
		\node at (2,0) {$v^3$};
		\node at (6,0) {$u v^2$};
		\node at (12,0) {$u(1-u-v)v$};
		\node at (18,0) {$u^2 v$};
		\node at (24,0) {$u^3$};
	\end{scope}
	\end{tikzpicture}
\end{center}

The structure of the paper is as follows. In Section~\ref{sect:notion}, we provide all definitions and count TSPPs with respect to their diagonal. In Section~\ref{sect:lem}, we provide a lemma that allows us to express $\A_n(u,v;\mathbf{x})$ using a determinant. In Section~\ref{sect:proof}, we use this other expression for $\A_n(u,v;\mathbf{x})$ to prove Theorem~\ref{thm: main result}.

\section{Modified balanced partitions, TSPPs and ASMs} \label{sect:notion}

The first important objects are modified balanced partitions. They are a variation of objects that appear in \cite[Ex. 6.19(bb) p. 223]{Stanley99}.

\begin{defi}
Let $\lambda$ be a partition $\lambda=(\lambda_1,\ldots, \lambda_{n})$, where we allow zero parts.
We call $\lambda$ a \emph{modified balanced} of size $n$ if
  $\lambda_1 \leq n-1$ and $\lambda_i < \lambda_i^\prime$ whenever $\lambda_i \geq i$, where $\lambda^\prime=(\lambda^\prime_1,\ldots, \lambda^\prime_{m})$ denotes the conjugate partition.
\end{defi}
The modified balanced partitions of $n=3$ are displayed next (using French notation), together with their Frobenius notation which we recall next.

\begin{center}
\begin{tikzpicture}
\begin{scope}[scale=0.5]
\node at (1.5,1) {$\emptyset$};
\node at (1.5,-1) {$\emptyset$};
\node at (1.5,-2.5) {$(|)$};
\begin{scope}[xshift=5cm]
\Tableaux{{1,1}}
\node at (1.5,-1) {$(1,1)$};
\node at (1.5,-2.5) {$(0|1)$};
\end{scope}
\begin{scope}[xshift=10cm]
\Tableaux{{1,1,1}}
\node at (1.5,-1) {$(1,1,1)$};
\node at (1.5,-2.5) {$(0|2)$};
\end{scope}
\begin{scope}[xshift=15cm]
\Tableaux{{2,1,1}}
\node at (2,-1) {$(2,1,1)$};
\node at (2,-2.5) {$(1|2)$};
\end{scope}
\begin{scope}[xshift=20cm]
\Tableaux{{2,2,2}}
\node at (2,-1) {$(2,2,2)$};
\node at (2,-2.5) {$(1,0|2,1)$};
\end{scope}
\end{scope}
\end{tikzpicture}
\end{center}

Let $\lambda$ be a partition. Unless otherwise specified, we denote by $l$ the side length of the Durfee square of $\lambda$, which is defined as the largest square that is contained in the Ferrers diagram, i.e., $l=\max_{i}\{\lambda_i \geq i\}$. The \emph{Frobenius notation} of $\lambda$ is then $(\lambda_1-1,\lambda_2-2,\ldots,\lambda_l-l|\lambda_1^\prime-1,\lambda_2^\prime-2,\ldots,\lambda_l^\prime-l)$. 
Using the Frobenius notation, a partition $\lambda=(a_1,\ldots,a_l|b_1,\ldots b_l)$ is a modified balanced partition if $a_i < b_i$ for $1 \leq i \leq l$.

Modified balanced partitions of size $n$ are enumerated by the $n$-th Catalan number $C_n=\frac{1}{n+1}\binom{2n}{n}$ which can be seen easily by the following bijection between modified balanced partitions and Dyck paths (represented by north and east steps)
\[
(a_1,\ldots,a_l| b_1,\ldots, b_l) \mapsto N^{b_l}E^{a_l+1}N^{b_{l-1}-b_l}E^{a_{l-1}-a_l} \cdots 
N^{b_{1}-b_2}E^{a_{1}-a_2}N^{n-b_1}E^{n-a_1-1},
\]
for non-zero partitions and $(|) \mapsto N^n E^n$. For an example see Figure \ref{fig: mod bal partitions and Dyck paths}.\\

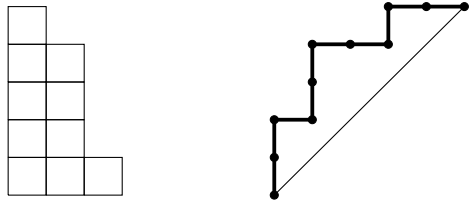
\begin{figure}[h]
\begin{center}
\begin{tikzpicture}
\begin{scope}[scale=0.5]
	\Tableaux{{3,2,2,2,1}}
	\begin{scope}[xshift=8cm, yshift=1cm]
		\draw[line width=0.5 mm] (0,0) -- (0,2) -- (1,2) -- (1,4) -- (3,4) -- (3,5) -- (5,5);
		\draw (0,0) -- (5,5);
		\filldraw (0,0) circle (3pt);
		\filldraw (0,1) circle (3pt);
		\filldraw (0,2) circle (3pt);
		\filldraw (1,2) circle (3pt);
		\filldraw (1,3) circle (3pt);
		\filldraw (1,4) circle (3pt);
		\filldraw (2,4) circle (3pt);
		\filldraw (3,4) circle (3pt);
		\filldraw (3,5) circle (3pt);
		\filldraw (4,5) circle (3pt);
		\filldraw (5,5) circle (3pt);
	\end{scope}
\end{scope}
\end{tikzpicture}
\end{center}
\caption{\label{fig: mod bal partitions and Dyck paths} The Ferrers diagram of the modified balanced partition $\lambda=(3,2,2,2,1)=(2,0|4,2)$ of size $5$ and its corresponding Dyck path.}
\end{figure}

Plane partitions are other combinatorial objects that are necessary in our study.

\begin{defi}
A \emph{plane partition} $\pi=(\pi_{i,j})_{1 \leq i,j \leq n}$ inside an $(n,n,n)$-box is an array of non-negative integers less than or equal to $n$ such that the rows and columns are weakly decreasing, i.e., $\pi_{i,j} \geq \pi_{i+1,j}$ and $\pi_{i,j} \geq \pi_{i,j+1}$.
\end{defi}
We can represent a plane partition $\pi$ graphically by replacing the $(i,j)$-th entry by a stack of $\pi_{i,j}$ unit cubes, see Figure \ref{fig: plane partition} for an example. From this point of view,  a \emph{plane partition} $\pi$ inside an $(n,n,n)$-box is a subset of $\{1, \ldots, n \}^3$ such that if $(i,j,k)$ is an element of $\pi$ then every $(i',j',k')$ with $i' \leq i,j' \leq j, k' \leq k$ is also an element of $\pi$. 

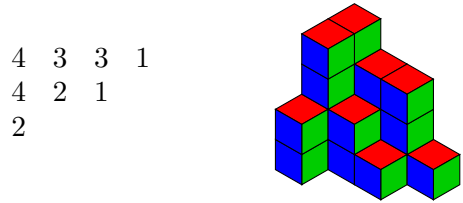
\begin{figure}[h]
	\centering
	\begin{tikzpicture}
		\begin{scope}
			\node at (0,0) {$\begin{array}{cccc}
				4 & 3 & 3 & 1\\
				4 & 2 & 1\\
				2
				\end{array}$};
		\end{scope}
		\begin{scope}[scale=0.4, xshift=9cm, yshift=-1cm]
			\PlanePartition{{4,3,3,1},{4,2,1,0},{2,0,0,0}}
		\end{scope}
	\end{tikzpicture}
	\caption{\label{fig: plane partition} A plane partition and its graphical representation as a pile of cubes.}
\end{figure}
A plane partition $\pi$ is \emph{totally symmetric} if for every $(i,j,k)$ that is an element of $\pi$, all permutations of the coordinates $(i,j,k)$ are also elements of $\pi$. We denote by $\TSPP_n$ the set of totally symmetric plane partitions (TSPPs) inside an $(n,n,n)$-box.
Given $T=(T_{i,j})_{1 \le i, j \le n-1} \in \TSPP_{n-1}$, we associate with $T$ a modified balanced partition\footnote{This can be regarded as a generalisation of Stanley's \cite{Stanley73} \emph{trace} statistic which is defined for a plane partition $\pi$ as the sum over the parts of its diagonal $\diag(\pi)$.} as follows: Consider the partition  $(T_{1,1},\ldots,T_{n-1,n-1})^\prime =: \diag(T)$, which is just the profile of the diagonal of $T$ in the $y=x$ plane if interpreted as stacks of cubes.  If $\diag(T)$ has Frobenius notation $(a_1,\ldots,a_l|b_1,\ldots,b_l)$, we set $\pi(T):=(a_1,\ldots,a_l|b_1+1,\ldots,b_l+1)$, for an example see Figure \ref{fig: TSPP and mod bal partition}. It can be checked that $\pi(T)$ is a modified balanced partition.\\

\begin{figure}[h]
\centering
\begin{tikzpicture}
\draw (0,-1.75) node {$T$};
\begin{scope}[scale=0.4]
\PlanePartition{{4,4,4,3},{4,3,2,1},{4,2,1,1},{3,1,1,0}}
\end{scope}
\begin{scope}[scale=0.5, xshift=7cm, yshift=-3cm]
\draw (2.5,-.5) node {$\diag(T)$};
\Tableaux{{3,2,2,1}}
\end{scope}
\begin{scope}[scale=0.5, xshift=14cm, yshift=-3cm]
\draw (2.75,-.5) node {$\pi(T)$};
\Tableaux{{3,2,2,2,1}}
\end{scope}
\end{tikzpicture}
\caption{\label{fig: TSPP and mod bal partition} A TSPP $T$ inside a $(4,4,4)$-box, its diagonal $\diag(T)$ and its associated modified balanced partition $\pi(T)$ of size $5$.}
\end{figure}
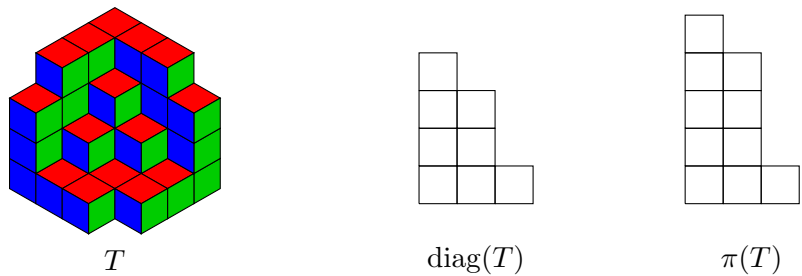

We count TSPPs $T$ with fixed $\pi(T)$. 

\begin{prop}
\label{prop: number of TSPP with given diagonal}
Let $\lambda=(a_1,\ldots, a_l|b_1,\ldots, b_l)$ be a modified balanced partition. The number of TSPPs $T$ with $\pi(T)=\lambda$ is equal to
$
\det_{1 \leq i,j \leq l} \left( \binom{b_j-1}{a_i} \right).
$
\end{prop}
\begin{proof}
This is a classical application of the Lindstr\"om-Gessel-Viennot theorem~\cite{GesselViennot85,Lindstroem73}, see also \cite{Stembridge95}. 
We sketch the proof on the example in Figure~\ref{fig:LGV}.  

TSPPs of order $n$ clearly correspond to lozenge tilings of a regular hexagon with side lengths $n$ that are symmetric with respect to the vertical symmetry axis as well as rotation of $120^\circ$. By this symmetry, it suffices to know a sixth of the lozenge tiling. In our example, we choose the sixth that is in the wedge of the red dotted rays.
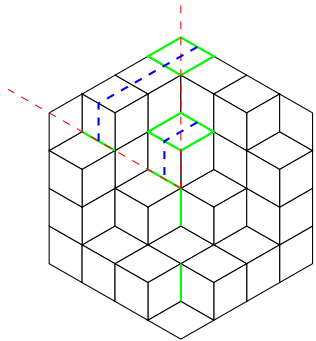
\begin{figure}[h]
\centering
\begin{tikzpicture}
\begin{scope}[scale=0.5]
\PlanePartitionWhite{{4,4,4,3},{4,3,2,1},{4,2,1,1},{3,1,1,0}}
\RightBoundary{0}{-4}{4}{3}
\LeftBoundary{4}{0}{4}{3}
\TopBoundary{4}{-4}{0}{3}
\draw [color=green, thick] (0,1) -- ({1*cos(30)},{1+1*sin(30)}) -- (0,2) -- ({-1*cos(30)},{1+1*sin(30)}) -- (0,1);
\draw [color=green, thick] (0,3) -- ({1*cos(30)},{3+1*sin(30)}) -- (0,4) -- ({-1*cos(30)},{3+1*sin(30)}) -- (0,3);
\draw [color=green, thick] (0,-1) -- (0,0);
\draw [color=green, thick] (0,-3) -- (0,-2);
\draw [color=green, thick] (0,0) -- ({-1*cos(30)},{1*sin(30)});
\draw [color=green, thick] ({-3*cos(30)},{3*sin(30)}) -- ({-2*cos(30)},{2*sin(30)});
\draw [dashed, red] (0,0) -- (0,5);
\draw [dashed, red] (0,0) -- ({-5.25*cos(30)},{5.25*sin(30)});
\draw [dashed, blue, thick] ({1/2*cos(30)},{1+3/2*sin(30)}) -- ({-1/2*cos(30)},{1+1/2*sin(30)})  -- ({-1/2*cos(30)},{1/2*sin(30)});
\draw [dashed, blue, thick] ({1/2*cos(30)},{1+11/2*sin(30)}) -- ({-5/2*cos(30)},{1+5/2*sin(30)}) -- ({-5/2*cos(30)},{5/2*sin(30)});
\end{scope}
\end{tikzpicture}
\caption{\label{fig:LGV} Running example in the proof of Proposition~\ref{prop: number of TSPP with given diagonal}.}
\end{figure} 

\noindent
Now observe that the positions of the horizontal lozenges in the upper half of the vertical symmetry axis are prescribed by the $b_i$'s, while the positions of the vertical segments in the lower part of the vertical symmetry axis are prescribed by the $a_i$'s. Both are indicated in green in Figure~\ref{fig:LGV}. By the cyclic symmetry, these green segments have corresponding segments on the red dotted ray that is not contained on the vertical symmetry axis, again indicated in green in the figure. Now the lozenge tiling is determined by the family of non-intersecting lattice paths that connect these segments with the horizontal lozenges in the upper half of the vertical symmetry axis, indicated in blue in the figure.
\end{proof}

The third objects of importance are alternating sign matrices.

\begin{defi}
An \emph{alternating sign matrix}, or \emph{ASM} for short, of size $n$ is an $n \times n$ matrix with entries $-1,0,1$ such that all row- and column-sums are equal to $1$ and in all rows and columns the non-zero entries alternate.
\end{defi}

It is easy to see that every ASM has a unique $1$ in its top row. A product formula for the refined enumeration of ASMs with respect to the position of the unique $1$ in the top row was conjectured by Mills, Robbins, Rumsey \cite{MillsRobbinsRumsey83} and first proven by Zeilberger \cite{Zeilberger96b}. Following the convention of \cite{Fischer18}, we define the \emph{inversion number} and \emph{complementary inversion number} of an ASM $A=(a_{i,j})_{1 \leq i,j \leq n}$ of size $n$ as
$$
\inv(A):= \sum_{1 \leq i^\prime < i \leq n \atop 1 \leq j^\prime \leq j \leq n}a_{i^\prime, j}a_{i,j^\prime} \quad \text{and} \quad
\inv^\prime(A):= \sum_{1 \leq i^\prime < i \leq n \atop 1 \leq j \leq j^\prime \leq n}a_{i^\prime, j}a_{i,j^\prime},
$$
and denote by $\mathcal{N}(A)$ the number of $-1$'s of $A$.
For instance,
\[
A= \begin{pmatrix}
0 & 1 & 0 & 0\\
1 & -1 & 0 & 1 \\
0 & 0 & 1 & 0 \\
0 & 1 & 0 & 0
\end{pmatrix}
\]
is an ASM of size $4$ with $(\mathcal{N}(A),\inv(A),\inv^\prime(A))=(1,3,2)$. 
The number of $-1$ entries, the inversion number and the complementary inversion number of an ASM $A$ of size $n$ are connected by
$
\mathcal{N}(A)+\inv(A)+\inv^\prime(A)=\binom{n}{2}.
$
The weighted enumeration of ASMs, where each ASM $A$ is weighted by $t^{\mathcal{N}(A)}$, is called the $t$-enumeration. For $t \in \{0,1,2,3\}$, it turns out that the $t$-enumeration is given by explicit product formulas, see for example \cite{Aigner_Det,Kuperberg96, MillsRobbinsRumsey83}.

\section{An antisymmetrizer-to-determinant lemma}  \label{sect:lem}

The following lemma will be a fundamental tool for the proof of Theorem~\ref{thm: main result}. More applications of it will appear in a forthcoming paper.

\begin{lem}  
\label{general}
Let $f(X),g(X)$ be Laurent series over $\mathbb{C}$ such that for every non-zero polynomial $p(X) \in \mathbb{C}[X]$ of degree no greater than $n$, we have $p(f(X)) \not= p(g(X))$. Then 
$$
\det_{1 \le i, j \le n} \left( f(X_i)^j - g(X_i)^j \right) 
 = \asym_{X_1,\ldots,X_n} \left[  \prod_{1 \le i \le j \le n} (f(X_j)-g(X_i))     \right].$$
\end{lem}

\begin{proof}  The proof is by induction with respect to $n$. The result is obvious for $n=1$. 
Let $L_n(X_1,\ldots,X_n), R_n(X_1,\ldots,X_n)$ denote the left and right hand side of the identity in the statement, respectively. By the induction hypothesis, we can assume 
$ L_{n-1}(X_1,\ldots,X_{n-1}) \allowbreak = R_{n-1}(X_1,\ldots,X_{n-1})$.
We show that both $L_n(X_1,\ldots,X_n)$ and $R_n(X_1,\ldots,X_n)$ can be computed recursively using 
$L_{n-1}(X_1,\ldots,X_{n-1})$ and $R_{n-1}(X_1,\ldots,X_{n-1})$, respectively, with the same recursion.
For the right hand side, we have 
$$
R_n(X_1,\ldots,X_n) 
=  \sum_{i=1}^{n} (-1)^{i+1}  \left( \prod_{k=1}^n (f(X_k)-g(X_i))  \right) R_{n-1}(X_1,\ldots,\widehat{X_i},\ldots,X_n),
$$
where $\widehat{X_i}$ means that $X_i$ is omitted.
For the left hand side, we first observe
\begin{equation}
\label{fundamentalidentity}
\sum_{j=0}^{n} (f(X_i)^j - g(X_i)^j) e_{n-j}(-f(X_1),\ldots,-f(X_n)) = (-1)^{n-1} \prod_{k=1}^{n} (f(X_k)-g(X_i)),
\end{equation}
where $e_{j}(X_1,\ldots,X_n)$ denotes the $j$-th elementary symmetric function. Note that the summand for $j=0$ on the left hand side is actually $0$. 
Now consider the following system of linear equations with $n$ unknowns $c_j(X_1,\ldots,X_n)$, $1 \le j \le n$, and 
$n$ equations.
$$
\sum_{j=1}^{n} (f(X_i)^j - g(X_i)^j)  c_j(X_1,\ldots,X_n) 
= (-1)^{n-1} \prod_{k=1}^{n} (f(X_k)-g(X_i)), \quad 1 \le i \le n.
$$
The determinant of this system of equations is obviously $L_n(X_1,\ldots,X_n)$, which is non-zero by the assumption. By \eqref{fundamentalidentity}, we know that the unique solution of this system is given by 
$
c_j(X_1,\ldots,X_n) = e_{n-j}(-f(X_1),\ldots,-f(X_n)).
$
On the other hand, by Cramer's rule, 
$$
c_n(X_1,\ldots,X_n) = \frac{\det \limits_{1 \le i, j \le n} \left( \begin{cases}  f(X_i)^j - g(X_i)^j, & \text{if $j<n$} \\
  (-1)^{n-1} \prod\limits_{k=1}^{n} (f(X_k)-g(X_i)) , & \text{if $j=n$}  \end{cases} \right)}{L_n(X_1,\ldots,X_n)}.
$$
The assertion now follows from $c_n(X_1,\ldots,X_n) = e_{0}(X_1,\ldots,X_n)=1$ and expanding the determinant in the numerator with respect to the last column.
\end{proof}

In order to apply the lemma to $\A_n(u,v;\mathbf{x})$, we first observe that it is equal to
$$
\prod\limits_{i=1}^n \left( \frac{ x_i^{n-1}}{\frac{v}{x_i} + (1-u-v) + ux_i} \right)
   \frac{ \asym_{x_1,\ldots,x_n} 
\left[  \prod\limits_{1 \le i \le  j \le n} \left(\frac{v}{x_i} + (1-u-v) + ux_j\right) \right]}{ \prod\limits_{1 \le i < j \le n} (x_j - x_i)}
$$
By the lemma, this is further equal to 
$\frac{ \det\limits_{1 \le i, j \le n} \left( x_i^{n-j} p_j(x_i) \right)}{\prod\limits_{1 \le i < j \le n} (x_i - x_j)}$
with $p_{j}(x):=  \sum\limits_{k=0}^{j-1} x^k (-1+u+v-ux)^k v^{j-1-k}$.

\section{Proof of Theorem~\ref{thm: main result}} \label{sect:proof}
The proof of Theorem~\ref{thm: main result} is split into two parts. First, we derive an explicit expansion of $A_n(u,v;\mathbf{x})$ into Schur polynomials. Second we prove that the coefficients of each Schur polynomial satisfy the same recursion as the right hand side of \eqref{eq: A-poly as sum of TSPPs}.

To emphasise the general principle used to express the determinantal expression of $\A_n(u,v;\mathbf{x})$ as a sum of Schur polynomials, we consider $p_j(x)$ to be a family of polynomials $p_j(x):= \sum_{k\geq 0} a_{j,k} x^k$.
Using the linearity of the determinant in the columns, we have
\begin{equation}
\label{eq: det to Schur I}
\frac{\det\limits_{1 \leq i,j \leq n}\left( x_i^{n-j}p_j(x_i) \right)}{\prod\limits_{1 \leq i< j \leq n}(x_i-x_j)}
= \sum\limits_{k_1,\ldots,k_n \geq 0} \left( \prod_{j=1}^n a_{j,k_j} \right) s_{(k_1,\ldots, k_n)}(\mathbf{x}),
\end{equation}
where we used in the last step the well known extension of Schur polynomials to arbitrary sequences $L=(L_1,\ldots,L_n)$ of non-negative integers via
$
s_{L}(\mathbf{x}) := \frac{\det\limits_{1 \leq i,j \leq n}\left( x_i^{L_j+n-j}\right)}{\prod\limits_{1 \leq i < j \leq n}(x_i-x_j)}.
$
It can be checked that the generalised Schur polynomial $s_L(\mathbf{x})$ is either equal to $0$ or $s_L(\mathbf{x})= \sgn(\sigma) s_\lambda(\mathbf{x})$ where $\lambda=(\lambda_1,\ldots, \lambda_n)$ is a partition whose parts are allowed to be zero and $\sigma \in S_n$ is a permutation such that $L_j=\lambda_{\sigma(j)}+j-\sigma(j)$ for all $1 \leq j \leq n$. It follows that \eqref{eq: det to Schur I} is equal to
\begin{equation}
\label{eq: det to Schur II}
 \sum_{\lambda} s_\lambda(\mathbf{x}) \left( \sum_{\sigma \in S_n} \sgn(\sigma) \prod_{j=1}^n a_{j,\lambda_{\sigma(j)}+j-\sigma(j)}  \right)
= \sum_{\lambda} s_\lambda(\mathbf{x}) \det\limits_{1 \leq i,j \leq n} \left( a_{j,\lambda_i+j-i} \right),
\end{equation}
where the sum is over all partitions $\lambda$.
By applying \eqref{eq: det to Schur II} to the family of polynomials 
$$
p_{j}(x) 
= \sum\limits_{0 \leq l, k \leq j-1}
(-1)^{k}\binom{k}{l} x^{k+l} u^{l}(1-u-v)^{k-l}v^{j-1-k},
$$
we obtain

\begin{multline*}
\A_n(u,v;\mathbf{x})=
\sum_\lambda s_\lambda(\mathbf{x}) \det\limits_{1 \leq i,j \leq n} \left( \sum\limits_{0 \leq l, k \leq j-1 \atop k+l = \lambda_i+j-i}
(-1)^{k}\binom{k}{l} u^{l}(1-u-v)^{k-l}v^{j-1-k} \right)\\
=\sum_\lambda s_\lambda(\mathbf{x})\\ \times \det_{1 \leq i,j \leq n}\left( \sum_{k=0}^{j-1}
(-1)^{k}\binom{k}{\lambda_i+j-i-k}
u^{\lambda_i+j-i-k}(1-u-v)^{2k-\lambda_i-j+i}v^{j-1-k}
\right).
\end{multline*}
We denote by $m_{i,j}(\lambda_i)$ the $(i,j)$-th entry of the matrix in the above determinant.
An entry $m_{i,1}(\lambda_i)= \binom{0}{\lambda_i+1-i}u^{\lambda_i+1-i}(1-u-v)^{-\lambda_i-1+i}$ in the first column is $1$ iff $\lambda_i=i-1$ and $0$ otherwise. Let $l$ be the side length of the Durfee square of $\lambda$. The only possible part of $\lambda$ satisfying $\lambda_i=i-1$ is the $(l+1)$-st. For $\lambda_{l+1} \neq l$ the partition $\lambda$ is not a modified balanced partition and the above determinant is $0$. Hence we assume for the rest of the proof $\lambda_{l+1}=l$. By expanding the determinant along the first column, we obtain
\[
\det_{1 \leq i,j \leq n} \left(m_{i,j}(\lambda_i) \right) = (-1)^{l+2}\det_{1 \leq i,j \leq n-1} \left(m_{i,j}^\prime(\lambda_i) \right),
\]
where $(m_{i,j}^\prime)_{1,\leq,i,j \leq n-1}$ denotes the matrix obtained by deleting the first column and the $(l+1)$-st row of $(m_{i,j}(\lambda_i))_{1\leq i,j \leq n}$. For $1 \leq i \leq l$, i.e., $\lambda_{i}>i$, we can rewrite $m_{i,j}^\prime$ as
\begin{multline*}
m_{i,j}^\prime =\\ \sum_{k=0}^{j}(-1)^{k}\binom{k}{\lambda_i+(j+1)-i-k}u^{\lambda_i+(j+1)-i-k}(1-u-v)^{2k-\lambda_i-(j+1)+i}v^{(j+1)-1-k}\\
= \sum_{k=0}^{j-1}(-1)^{k+1}u^{\lambda_i+j-i-k}(1-u-v)^{2k+1-\lambda_i-j+i}v^{j-k-1}\\
\times \left( \binom{k}{\lambda_i+j-i-k}+ \binom{k}{\lambda_i+j-i-k-1} \right) \\
=-(1-u-v) m_{i,j}(\lambda_i)-u m_{i,j}(\lambda_i-1).
\end{multline*}
For $i>l$ on the other hand, i.e. $\lambda_{i+1}<i+1$, we can express $m_{i,j}^\prime$ analogously as
\begin{multline*}
m_{i,j}^\prime = \sum_{k=0}^{j}(-1)^{k}\binom{k}{\lambda_{i+1}+(j+1)-(i+1)-k}\\
\times u^{\lambda_{i+1}+(j+1)-(i+1)-k}(1-u-v)^{2k-\lambda_{i+1}-(j+1)+(i+1)}v^{(j+1)-1-k} 
=v m_{i,j}(\lambda_{i+1}),
\end{multline*}
since $\binom{j}{\lambda_{i+1}-i}=0$.
The coefficient $c_{\lambda}$ of $s_\lambda(\mathbf{x})$ in $\A_{n}(u,v;\mathbf{x})$ is therefore given by
\begin{multline*}
(-1)^{l}\det_{1 \leq i,j \leq n-1} \left( 
\begin{cases}
-(1-u-v) m_{i,j}(\lambda_i)- u m_{i,j}(\lambda_i-1) \quad & i \leq l,\\
v m_{i,j}(\lambda_{i+1}) & i>l,
\end{cases} \right) \\
  = \sum_{(f_1,\ldots,f_l) \in \{0,1\}^l} \left(u^{\sum_{i=1}^l f_i} (1-u-v)^{l-\sum_{i=1}^l f_i} v^{n-1-l}\right)
  c_{(\lambda_1 -f_1,\ldots,\lambda_l-f_l,\lambda_{l+2},\ldots,\lambda_n)},
\end{multline*}
with $c_{(\lambda_1 -f_1,\ldots,\lambda_l-f_l,\lambda_{l+2},\ldots,\lambda_n)}=0$ if $(\lambda_1 -f_1,\ldots,\lambda_l-f_l,\lambda_{l+2},\ldots,\lambda_n)$ is not a partition, where the equality follows from the linearity of the determinant in the rows and choosing $f_i=0$ iff we select the first term in row $i$.
Using Frobenius notation for $\lambda=(a_1,\ldots,a_l|b_1,\ldots,b_l)$, the above recursion can be rewritten as
\begin{multline*}
c_{(a_1,\ldots,a_l|b_1,\ldots,b_l)} =\\ \sum_{(f_1,\ldots,f_l) \in \{0,1\}^l}\left(u^{\sum_{i=1}^l f_i} (1-u-v)^{l-\sum_{i=1}^l f_i} v^{n-1-l}\right) c_{(a_1-f_1,\ldots,a_l-f_l|b_1-1,\ldots,b_l-1)},
\end{multline*}
where $c_{(a_1,\ldots,a_{l-1},-1|b_1,\ldots,b_{l-1},0)}$ is defined as $c_{(a_1,\ldots,a_{l-1}|b_1,\ldots,b_{l-1})}$ .

Denote by $d_\lambda$  the coefficient of $s_\lambda(\mathbf{x})$ in $\sum_{T \in \TSPP_{n-1}} \omega_{\pi(T)}(u,v) s_{\pi(T)}(\mathbf{x})$. For $\lambda=(a_1,\ldots,a_l|b_1,\ldots,b_l)$, Proposition \ref{prop: number of TSPP with given diagonal} implies 
\begin{multline*}
d_{(a_1,\ldots,a_l|b_1,\ldots,b_l)} =\\ u^{\sum_{i=1}^l (a_i+1)} (1-u-v)^{\sum_{i=1}^l (b_i-1-a_1)} v^{\binom{n}{2}-\sum_{i=1}^lb_i} \det_{1 \leq i,j \leq l} \left( \binom{b_j-1}{a_i} \right)\\
= u^{\sum_{i=1}^l (a_i+1)} (1-u-v)^{\sum_{i=1}^l (b_i-1-a_1)} v^{\binom{n}{2}-\sum_{i=1}^lb_i} \det_{1 \leq i,j \leq l} \left( \binom{b_j-2}{a_i}+\binom{b_j-2}{a_i-1} \right)\\
= \sum_{(f_1,\ldots,f_l) \in \{0,1\}^l} \left(u^{\sum_{i=1}^l f_i} (1-u-v)^{l-\sum_{i=1}^l f_i} v^{n-1-l} \right) d_{(a_1-f_1,\ldots,a_l-f_l|b_1-1,\ldots,b_l-1)},
\end{multline*}
where we used the linearity of the determinant in the last step. The assertion follows by induction on $n$ since both $c_\lambda$ and $d_\lambda$ satisfy the same recursion and the induction base can be checked easily by hand.

\section{Final remarks}
We conclude our article with some brief remarks describing further directions that we intend to pursue.
Drawing upon work of Fischer and Riegler \cite{FischerRiegler15} inspired by counting monotone triangles with a fixed bottom row,
one can obtain a family of symmetric polynomials indexed by a partition $\lambda$ that contains $\mathcal{A}_n(u,v;\mathbf{x})$ as a special case of the empty partition.
Data reveal that this family still has a nice Schur expansion in the sense that the coefficients can be written as polynomials in $u,v,1-u-v$ with positive coefficients.
A next natural step would be to generalize our techniques to a broader framework.

Continuing in this direction, one may consider stable limits of the aforementioned symmetric polynomials to obtain (inhomogeneous) elements in the ring of symmetric functions. The lowest degree symmetric function in these expressions is the Schur function $s_{\lambda}$. This is very reminiscent of the stable Grothendieck polynomials of Fomin-Kirillov, and raises the question whether there is a combinatorially interesting algebra structure on these functions.

Finally, it is interesting to note that  $\A_n(u,v;\mathbf{x})$ already appeared in a six-vertex model context. More concretely,  
$X_n(x,y;z_1,\ldots,z_n)$ as defined by Behrend in \cite[Eq. (70)]{Behrend13} seems to satisfy $v^{\binom{n}{2}} X_n(\frac{u}{v},\frac{1}{v};\mathbf{x}) = \A_n(u,v;\mathbf{x})$. We plan to examine this connection in more detail in the full version of this extended abstract.

\section*{Acknowledgements}
The authors want to thank Fran\c{c}ois Bergeron for helpful discussions as well as two anonymous referees for helpful suggestions.

\bibliographystyle{abbrv}
\bibliography{SYM_ASM_arxiv}

\end{document}